\documentclass{article}
\usepackage[utf8]{inputenc}
\usepackage{fullpage}
\usepackage[numbers,sort,compress]{natbib}

\usepackage[utf8]{inputenc}
\usepackage{microtype}
\usepackage{graphicx}
\usepackage{enumitem}
\usepackage{booktabs} 
\usepackage{multirow}
\usepackage{adjustbox}
\usepackage{url}

\usepackage{amsfonts}       
\usepackage{nicefrac}       
\usepackage{microtype}      
\usepackage{lipsum}
\usepackage{amsmath}
\usepackage{amssymb}
\usepackage{changepage}
\usepackage{enumitem}
\setlist{leftmargin=5.5mm}

\usepackage{stmaryrd}
\usepackage{graphicx}
\usepackage{caption}
\usepackage{textcomp} 
\usepackage{quoting}
\usepackage{titletoc} 
\usepackage{fontawesome} 
\usepackage{wrapfig}
\usepackage{xspace}
\usepackage{ifthen} 

\usepackage{amsthm,amssymb}

\usepackage{pifont}

\PassOptionsToPackage{colorlinks = true,
            linkcolor = blue,
            urlcolor  = blue,
            citecolor = blue,
            anchorcolor = blue}{hyperref}

\usepackage{algorithm}
\usepackage{algorithmicx}
\usepackage{algpseudocode}
\PassOptionsToPackage{noend}{algpseudocode}
\algrenewcommand\algorithmicrequire{\textbf{Input:}}
\algrenewcommand\algorithmicensure{\textbf{Output:}}

\newlength{\continueindent}
\usepackage{etoolbox}
\makeatletter
\newcommand*{\ALG@customparshape}{\parshape 2 \leftmargin \linewidth \dimexpr\ALG@tlm+\continueindent\relax \dimexpr\linewidth+\leftmargin-\ALG@tlm-\continueindent\relax}
\apptocmd{\ALG@beginblock}{\ALG@customparshape}{}{\errmessage{failed to patch}}
\makeatother

\algblock{ParFor}{EndParFor}
\algnewcommand\algorithmicpardo{\textbf{in parallel do}}
\algrenewtext{ParFor}[1]{\algorithmicfor\ #1\ \algorithmicpardo}
\algrenewtext{EndParFor}{\textbf{end for}}
\algtext*{EndParFor}
\algdef{SE}[SUBALG]{Indent}{EndIndent}{}{\algorithmicend\ }%
\algtext*{Indent}
\algtext*{EndIndent}

\usepackage{cleveref}
\usepackage{listings}

\usepackage{tikz,pgfplots}
\usepackage{pgfplotstable}
\usetikzlibrary{shapes}
\usetikzlibrary{positioning}
\usetikzlibrary{plotmarks}
\usetikzlibrary{patterns}
\usetikzlibrary{intersections,shapes.arrows}
\usetikzlibrary{arrows.meta}
\usetikzlibrary{pgfplots.fillbetween}
\usepgfplotslibrary{groupplots}
\usetikzlibrary{fit}

\definecolor{C0}{HTML}{1F77B4}
\definecolor{C1}{HTML}{ff7f0e}
\definecolor{C2}{HTML}{2ca02c}
\definecolor{C3}{HTML}{d62728}
\definecolor{C4}{HTML}{9467bd}

\definecolor{cadmiumgreen}{rgb}{0.0, 0.42, 0.24}
\definecolor{oldmauve}{rgb}{0.4, 0.19, 0.28}
\definecolor{royalazure}{rgb}{0.0, 0.22, 0.66}
\definecolor{harvardcrimson}{rgb}{0.79, 0.0, 0.09}
\definecolor{lightmauve}{rgb}{0.86, 0.82, 1.0}
\definecolor{darkbrown}{rgb}{0.4, 0.26, 0.13}
\definecolor{darkred}{rgb}{0.8, 0.0, 0.0}

\usepackage{cleveref}
\crefname{section}{Sec.}{Sec.}
\crefname{appendix}{App.}{App.}

\crefname{theorem}{Prop.}{Thm.}
\crefname{lemma}{Lem.}{Lem.}
\crefname{proposition}{Prop.}{Prop.}
\crefname{definition}{Def.}{Def.}
\crefname{assumption}{Asmp.}{Asmp.}

\Crefname{algorithm}{Alg.}{Alg.}
\crefname{figure}{Fig.}{Fig.}
\crefname{table}{Tab.}{Tab.}
\usepackage{bm,mathtools}
\usepackage{bbm}
\usepackage{lscape}

\newcommand{\myparagraph}[1]{\vspace*{0.5em}\par\noindent\textbf{{#1}.}} 

\renewcommand{\epsilon}{\varepsilon}

\newcommand{\tabemph}[1]{}

\usepackage{color}
\definecolor{puorange}{rgb}{0.80,0.20,0}
\definecolor{bluegray}{rgb}{0.04,0,0.7}
\definecolor{greengray}{rgb}{0.05,0.50,0.15}
\definecolor{darkbrown}{rgb}{0.40,0.2,0.05}
\definecolor{darkcyan}{rgb}{0,0.4,1}
\definecolor{black}{rgb}{0,0,0}
\definecolor{grey}{rgb}{0.93,0.93,0.93}

\newcommand \reals {\mathbb{R}}

    \newcommand \T {^{\top}}	
    \newcommand{\x}{w}
    
\DeclarePairedDelimiterX{\inp}[2]{\langle}{\rangle}{#1, #2} 
\DeclarePairedDelimiterX{\normsq}[1]{\Vert}{\Vert^2}{#1} 

\newcommand \grad {\nabla}

\newtheorem{theorem}{Theorem}

\newtheorem{proposition}[theorem]{Proposition}

\usepackage{algorithm}
\usepackage{algorithmicx}
\usepackage[noend]{algpseudocode}
\algnewcommand{\Initialize}[1]{%
  \State \textbf{Initialize:}
  \Statex \hspace*{\algorithmicindent}\parbox[t]{.8\linewidth}{\raggedright #1}
}

\DeclareMathOperator*{\argmin}{arg\,min}

\newcommand{\norm}[1]{\left\|#1 \right\|}

\newcommand{\innerfunc}{\phi}
\newcommand{\outerfunc}{f}

\newcommand{\varin}{w}
\newcommand{\varout}{u}
\newcommand{\spaceout}{U}
\newcommand{\subgrad}{v}
\newcommand{\dimin}{d}
\newcommand{\dimout}{k}
\newcommand{\nsamp}{n}

\newcommand{\sharparam}{\mu}
\newcommand{\surj}{\nu}

\newcommand{\dist}{\mathrm{dist}}

\newcommand{\stepsize}{\gamma}

\newcommand{\uv}{u}
\newcommand{\wv}{w}
\newcommand{\vv}{v}
\newcommand{\xv}{x}
\newcommand{\yv}{y}
\newcommand{\gv}{g}
\newcommand{\phiv}{\phi}
\newcommand{\xiv}{\xi}

\newcommand \id I

\begin{document}

\title{Modified Gauss-Newton Algorithms under Noise
\thanks{
This work was supported by NSF DMS-2023166, NSF CCF-2019844, NSF DMS-2052239, NSF DMS-2134012, NSF DMS-2133244, NIH, CIFAR-LMB, and faculty research awards.  \\
$^{\star\dagger}$ Now at Google Research.
}
}

\author{
Krishna Pillutla$^{1\star}$
$\qquad$
Vincent Roulet$^{1\dagger}$ 
$\qquad$
Sham Kakade$^2$
$\qquad$
Zaid Harchaoui$^1$ \\
{\small 
$^1$ University of Washington 
$\qquad$
$^2$ Harvard University
}
}

\date{\vspace{-1em}}

\maketitle

\begin{abstract}
    Gauss-Newton methods and their stochastic version have been widely used in machine learning and signal processing. Their nonsmooth counterparts, modified Gauss-Newton or prox-linear algorithms, can lead to contrasting outcomes when compared to gradient descent in large-scale statistical settings. We explore the contrasting performance of these two classes of algorithms in theory on a stylized statistical example, and experimentally on learning problems including structured prediction. In theory, we delineate the regime where the quadratic convergence of the modified Gauss-Newton method  is active under statistical noise. In the experiments, we underline the versatility of stochastic (sub)-gradient descent to minimize nonsmooth composite objectives. 
\end{abstract}

\section{Introduction} \label{sec:intro}

Arising from the literature on non-linear least squares~\cite{nocedal2006numerical, bjorck1996numerical},
the Gauss-Newton method was proposed to tackle generic compositional problems of the form
$\min_{\varin \in \reals^d} \outerfunc(\innerfunc(\varin))$
by linearizing the inner function $\innerfunc$ around the current iterate and solving the resulting subproblem~\cite{burke1985descent}. 

The Gauss-Newton method and its variants such as the Levenberg-Marquardt method~\cite{levenberg1944method,nesterov2007modified} have been applied successfully in phase retrieval~\cite{cichocki2002adaptive,herring2019gauss, repetti2014nonconvex}, nonlinear control~\cite{sideris2005efficient, roulet2019iterative}, and non-negative matrix factorization~\cite{huang2019low}. 
Modern machine learning problems such as deep learning possess a similar compositional structure, which makes Gauss-Newton-like algorithms potential good candidates~\cite{drusvyatskiy2019efficiency,tran2020stochastic, zhang2020stochastic}. 
However, in such problems, we are often interested in the generalization performance on unseen data. It is unclear whether the additional cost of solving the subproblems can be amortized by the superior efficiency of Gauss-Newton-like algorithms.

In this paper, we investigate whether modified Gauss-Newton methods or prox-linear algorithms with incremental gradient inner loops are superior to direct stochastic subgradient algorithms for nonsmooth problems with a compositional objective and a finite-sum structure in terms of generalization error. We present a statistical example and quantify when the quadratic convergence of the exact prox-linear method is not active before hitting the noise level of the problem. 
We present synthetic experiments that delineate the regimes where the stochastic subgradient methods outperform the prox-linear method. We also compare these algorithms on a structured prediction problem with a convolutional neural network (end-to-end path planning). 
Experimental results suggest that modified Gauss-Newton methods or prox-linear algorithms offer marginal gains in some settings, and confirm the versatility of direct stochastic subgradient algorithms to tackle complex learning problems. 
All proofs are given in the appendix.

\section{Problem Setting and Optimization Algorithms} \label{sec:pl:bg}
Given $\innerfunc_i:\reals^\dimin \rightarrow \reals^\dimout$ smooth, and $\outerfunc: \reals^\dimout \rightarrow \reals$ convex and Lipschitz, we consider finite-sum compositional minimization problems of the form 
\begin{equation}
\label{eq:pl:erm_pb}
F(\varin) :=  \frac{1}{\nsamp} \sum_{i=1}^\nsamp \outerfunc(\innerfunc_i(\varin))
\end{equation}
For \emph{multi-output regression} of
 input $\xv_i \in \reals^p$ to output $\yv_i \in \reals^\dimout$, we take $\phiv_i(\varin) = \varphi(\xv_i; \varin) - \yv_i$ as the residual of a predictor $\varphi(\cdot \,; \varin)$. We take a nonsmooth loss function such as $\outerfunc(\uv) = \|\uv\|_2$ ($\ell_2$ loss without the square), which is applicable in robust regression problems. 
A more sophisticated example is \emph{structured prediction}, the prediction of a combinatorial object such as a sequence. Here, $\innerfunc_i(\wv)$ is a score for each structured output, and $f$ is the structural hinge loss~\cite{crammer2001algorithmic,taskar2004max,tsochantaridis2004support}, computed efficiently using dynamic programming~\cite{viterbi1967error}. For applications, see, e.g.,~\cite{rush-2020-torch,gales2012structured,ratliff2006maximum}.

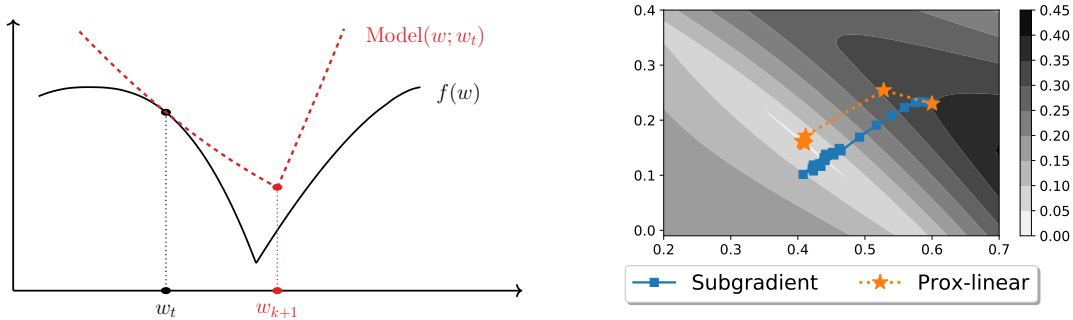
\begin{figure}[t]
\centering
\begin{adjustbox}{max width=0.42\linewidth}
\begin{tikzpicture}[scale=4, xscale=1.5]
    \draw[->, very thick] (-0.1, -1.2) --  (-0.1, 0.4);
    \draw[->, very thick] (-0.1, -1.2) --  (1.9, -1.2);

    \draw[very thick, domain=0:0.8541, smooth, black] plot({\x}, {-3 * (abs(\x - 0.2))^2.5});
    \draw[very thick, domain=0.8541:1.5, smooth, black] plot({\x}, {-2 * (abs(\x - 1.5))^1.5});
    
    \filldraw[black] (0.5, -0.147885090526395) circle [radius=0.5pt];
    \filldraw[black] (0.5, -1.2) circle [radius=0.5pt];
    \draw[thick, dotted]  (0.5, -0.147885090526395) -- (0.5, -1.2); 
    \node at (0.5, -1.32) {\Large ${w}_t$} ;
    
    \draw[ultra thick, domain=0.16:0.93762, smooth, dashed, C3] plot({\x}, {0.468302786666917 - 1.23237575438662 * \x + 0.5 * 1 * (\x - 0.5)^2});
    \draw[ultra thick, domain=0.93762:1.2, smooth, dashed, C3] plot({\x}, {3 * \x - 3.5 + 0.5 * 1 * (\x - 0.5)^2});
    
    \node at (1.65, -0.04) {\Large $f({w})$} ;
    \node[C3] at (1.52, 0.3) {\Large $\text{Model}(w; w_t)$} ;
    
    \filldraw[C3] (0.93762, -0.59) circle [radius=0.5pt];
    \filldraw[C3] (0.93762, -1.2) circle [radius=0.5pt];
    \draw[thick, dotted, C3]  (0.93762, -0.59) -- (0.93762, -1.2); 
    \node[C3] at (0.93762, -1.32) {\Large ${w}_{k+1}$} ;
    
\end{tikzpicture}
\end{adjustbox}
\hspace{2em}
\includegraphics[width=0.4\linewidth]{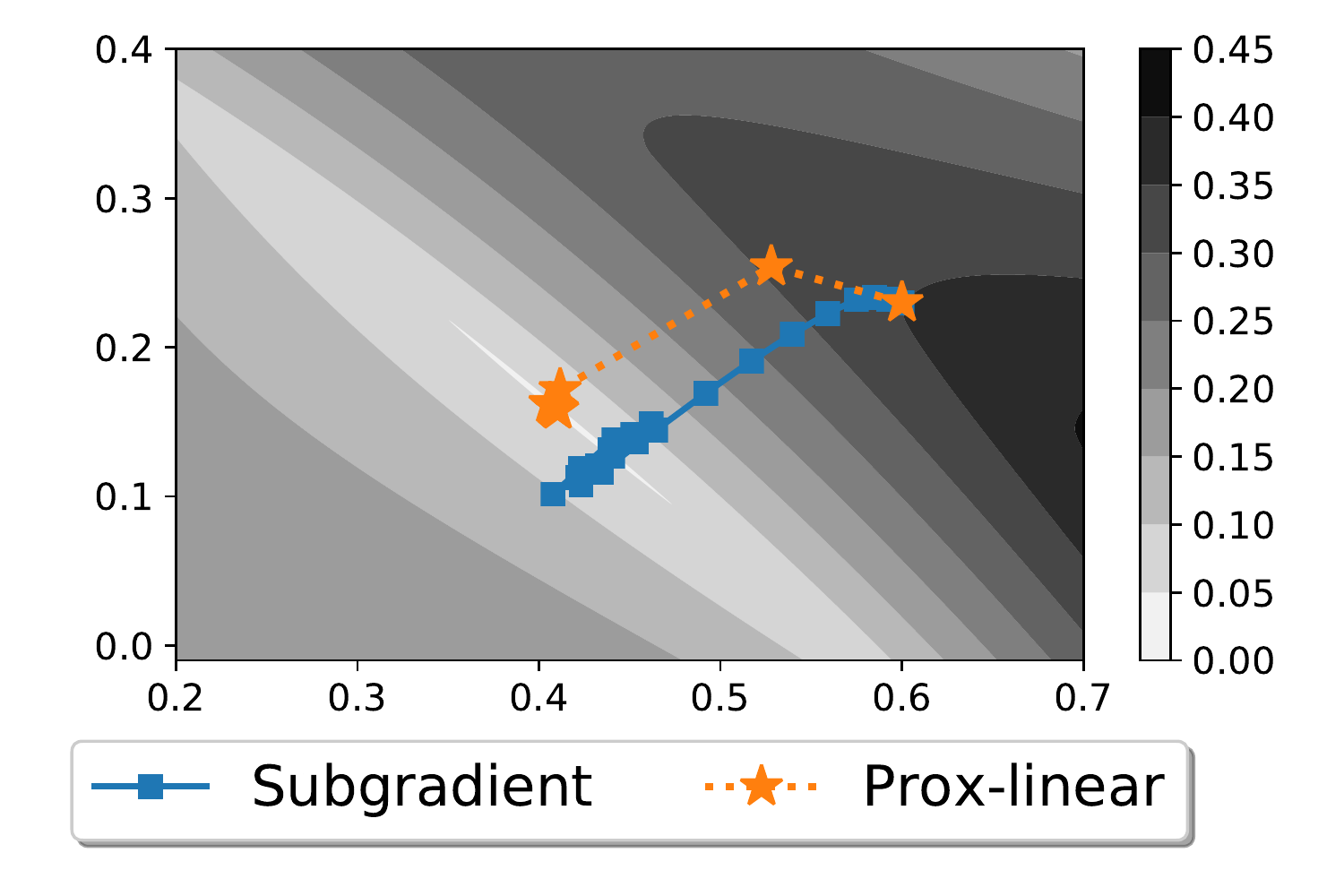}
\caption{
\small
\textbf{Left}: A numerical comparison of gradient descent and the exact prox-linear method on a nonsmooth and nonconvex function $F: \reals^2 \to \reals$. The prox-linear method builds a more accurate model of the objective function, especially around points of nonsmoothness.
\textbf{Right}:
    The modified Gauss-Newton (a.k.a. the prox-linear) method builds a convex model of the objective $F(\varin)$ around $\varin_t$ and finds $\varin_{t+1}$ by minimizing this model.
}
\label{fig:pl:gd-vs-pl}
\end{figure}

We compare two families of optimization algorithms, which are stochastic and nonsmooth versions 
of gradient descent and the Gauss-Newton algorithm. 
The \textbf{stochastic subgradient method} (abbreviated \textit{SGD}) is the nonsmooth and stochastic analogue of gradient descent.
In iteration $t$, SGD samples an element $i_t$ from the available $\nsamp$ uniformly at random and takes a step 
in the direction of its subgradient $\subgrad_t \in \partial (\outerfunc \circ \innerfunc_{i_t})(\varin_t)$: 
\begin{align}
    \varin_{t+1} = \varin_t - \gamma \subgrad_t \,,
\end{align}
where $\gamma$ is the learning rate and
$\partial (\outerfunc \circ \innerfunc_{i_t})$ denotes the regular (or Fr\'echet) subdifferential.
For \eqref{eq:pl:erm_pb}, it takes a simple form
$\partial (\outerfunc \circ \innerfunc_i)(\varin) = \grad \innerfunc_i(w)\T \partial \outerfunc \big(\innerfunc_i(w)\big)$, 
where $\partial \outerfunc$ refers to the convex subdifferential of $\outerfunc$ and $\grad \innerfunc_i$ refers to the Jacobian of $\innerfunc_i$~\cite[Theorem 10.6]{rockafellar2009variational}.
In deep learning, the subgradient $\subgrad \in \partial(\outerfunc \circ \innerfunc_i(\varin))$ requires the computation of the vector-Jacobian product, readily given by reverse-mode automatic differentiation implemented in software such as PyTorch.

The \textbf{modified Gauss-Newton} method~\cite{nesterov2007modified}, also known as the \textbf{prox-linear} method~\cite{burke1985descent,burke1995gauss, drusvyatskiy2019efficiency}, applied to~\eqref{eq:pl:erm_pb}, proceeds by finding approximate solutions of a partially linearized approximation of the objective with an additional regularization term. 
Given 
a regularization parameter $\kappa\geq 0$,
it iterates
\begin{align}
\varin_{t+1} = \argmin_{\varin \in \reals^\dimin} \  & 
\Bigg\{
\frac{1}{n} \sum_{i=1}^\nsamp \outerfunc(\innerfunc_i(\varin_t)  + \nabla \innerfunc_i(\varin_t)^\top(\varin - \varin_t)) 
+ \frac{\kappa}{2} \|\varin-\varin_t \|_2^2
\Bigg\}  \,.
\label{eq:pl:prox_lin}
\end{align}

As explained in \Cref{fig:pl:gd-vs-pl} (left), the prox-linear method creates a \textit{convex model} of $F$ around $\varin_t$ by linearizing the inner function as 
$\innerfunc_i(\varin) \approx \innerfunc_i(\varin_t)  + \nabla \innerfunc_i(\varin_t)^\top(\varin - \varin_t)$.
The next iterate \eqref{eq:pl:prox_lin} is the minimizer 
of the model plus a proximal term. 
Compare this with the subgradient method, where the model is 
$F(\varin_t) + \vv\T(\varin - \varin_t)$, for $\vv \in \partial F(\varin_t)$; see also \Cref{fig:pl:gd-vs-pl} (right).
It is usually not possible to solve the subproblem~\eqref{eq:pl:prox_lin} exactly (barring some special cases). We consider using accelerated incremental algorithms such as Casimir-SVRG~\cite{pillutla2018smoother}.
Computationally, each iteration of the inner loop requires having access to Jacobian-vector product $\vv\mapsto \grad \innerfunc_i(\varin) \vv$ which is most efficiently computed via forward-mode automatic differentiation.

\section{Tradeoffs of the Prox-linear Method in Statistical Settings} \label{sec:pl:noise}
Gauss-Newton methods and their variants are known to enjoy quadratic local convergence, provided \emph{the subproblems are solved exactly}. 
In statistical learning problems, it is not meaningful to optimize beyond the noise level of the problem. If the noise level of the problem is large, the quadratic convergence may not be useful. We formalize this in a stylized example.

We start with a typical quadratic local convergence result of the prox-linear method in the overparameterized regime $d > nk$. In particular, we assume that the minimal singular value $\sigma_{\min}(\grad \innerfunc(\varin)\T)$ of the transposed Jacobian of $\innerfunc = (\innerfunc_1; \cdots; \innerfunc_n)$ is strictly positive, which implies that the Jacobian $\grad \innerfunc$ is surjective. 

\begin{proposition} \label{prop:pl:main}
	Consider problem \eqref{eq:pl:erm_pb} where $\outerfunc$ is $\ell$-Lipschitz, convex, and $\sharparam$-sharp for some $\sharparam > 0$ 
 (see \Cref{sec:a:cas:quadratic} for a precise definition).
	 Suppose the function $\innerfunc(\varin) = (\innerfunc_1(\varin); \dots; \innerfunc_n(\varin)) \in \reals^{n \dimout}$ is $L$-smooth and satisfies $\sigma_{\min}(\grad \innerfunc(\varin)\T) \geq \surj >0$ for any $\varin \in \reals^\dimin$. 
        Then, the sequence $(\varin_t)_{t=0}^\infty$ produced by the exact prox-linear algorithm~\eqref{eq:pl:prox_lin} with $\kappa = L \ell$ starting from arbitrary $\varin_0 \in \reals^\dimin$ converges 
	    globally to its minimum value $F(\varin_t) \to  F^* := \min F$.
	    Further, as soon as an iterate $\varin_{j}$ satisfies $F(\varin_j) - F^* \leq (\sharparam\surj)^2/(L\ell n^{3/2})
    	$, the subsequence $(\varin_t)_{t=j}^\infty$ converges quadratically as
    	\begin{align} \label{eq:pl:prox-lin-quadratic}
            F(\varin_{t+1}) - F^*
            \leq \frac{L\ell n^{3/2}}{2(\sharparam\surj)^2}\big(F(\varin_t)
            - F^*\big)^2  \,.
    	\end{align}
\end{proposition}

\myparagraph{Statistical Setting} 
Suppose we are given $n$ input-output pairs $(\xv_i, \yv_i)$, where $\yv_i = \varphi(\xv_i; \overline \varin) + \xiv_i$ is given
from a parameterized nonlinear function $\varphi(\, \cdot \, ; \overline \varin): \reals^p \to \reals^\dimout$ with parameters $\overline \varin \in \reals^\dimin$, and $\xiv_i \sim \mathcal{N}(0, \sigma^2 \id_\dimout)$ is i.i.d. Gaussian noise. 
We wish to recover the true signal $\overline \varin$'s predictions, i.e., find $\hat \varin$ such that
$\varphi(\xv_i ; \hat \varin) \approx \varphi(\xv_i ; \overline \varin)$ for each $i$. 
We instantiate \eqref{eq:pl:erm_pb} with $\innerfunc_i(\varin) = \varphi(\xv_i ; \varin) - \yv_i$ and $\outerfunc = \| \cdot \|_2$ in order to solve this problem. 
This is different from the related choice of $\outerfunc = \norm{\cdot}_2^2$, which corresponds to non-linear least squares regression in the fixed design setting.

We consider \emph{early stopping} of the optimization once we reach the noise level. That is, we stop the optimization once the objective value $F(\varin_t)$ in iteration $t$ falls below $F(\overline \varin)$. 
In this setting, we now show that the prox-linear method can enjoy quadratic local convergence only when the noise level $\sigma$ of the problem is small enough. 
To this end, we make a general assumption on the radius $R$ of local quadratic convergence;  \Cref{prop:pl:main} provides a concrete lower bound on $R$.

\begin{proposition} \label{prop:pl:noise}
    Fix some $\delta \in (0, 1)$ and consider problem \eqref{eq:pl:erm_pb} with $\innerfunc_i$ and $\outerfunc$ as defined above with the output dimension $\dimout \ge 4 \log (2n/\delta)$. Suppose that
    (i)
    $\varin \mapsto \varphi(\xv_i; \varin)$ is $L$-smooth for each $i \in [n]$, and, 
    (ii) $\varphi$ can interpolate the data so that $\varphi(\xv_i; \wv^*) = \yv_i$ for each $i \in [n]$ for some $\varin^* \in \reals^\dimin$.
    Suppose there exists a scalar $R>0$ and an integer $j$ such that for all integers $t \ge j$, we have
    \begin{align} \label{eq:pl:noise-quadratic}
    \begin{aligned}
        F(\wv_t) - \min F &\leq R\,, \quad
        \text{and} \\
        F(\wv_{t+1}) - \min F &\leq \frac{1}{2R}\big(F(\wv_t) - \min F\big)^2  \,.
    \end{aligned}
    \end{align}
    Then, we have the following with probability at least $1-\delta$. If the noise level satisfies $\sigma > \tilde O\big(R/(k^{1/2} - k^{1/4})\big)$,
    then the first iterate $\varin_t$ enjoying quadratic convergence \eqref{eq:pl:noise-quadratic} 
    satisfies $F(\varin_t) < F(\overline \varin)$.
    Conversely, if the noise level satisfies
    $\sigma < \tilde O\big( R / (k^{1/2} + k^{1/4})\big)$, 
    then the first iterate $\varin_t$ enjoying quadratic convergence \eqref{eq:pl:noise-quadratic} 
    satisfies $F(\varin_t) > F(\overline \varin)$. 
\end{proposition}
\Cref{prop:pl:noise} shows that the potential advantages of the prox-linear method in terms of local quadratic convergence may not be relevant in some statistical problems with high noise.

\begin{figure}[t]
	\begin{center}
		\includegraphics[width=0.9\linewidth]{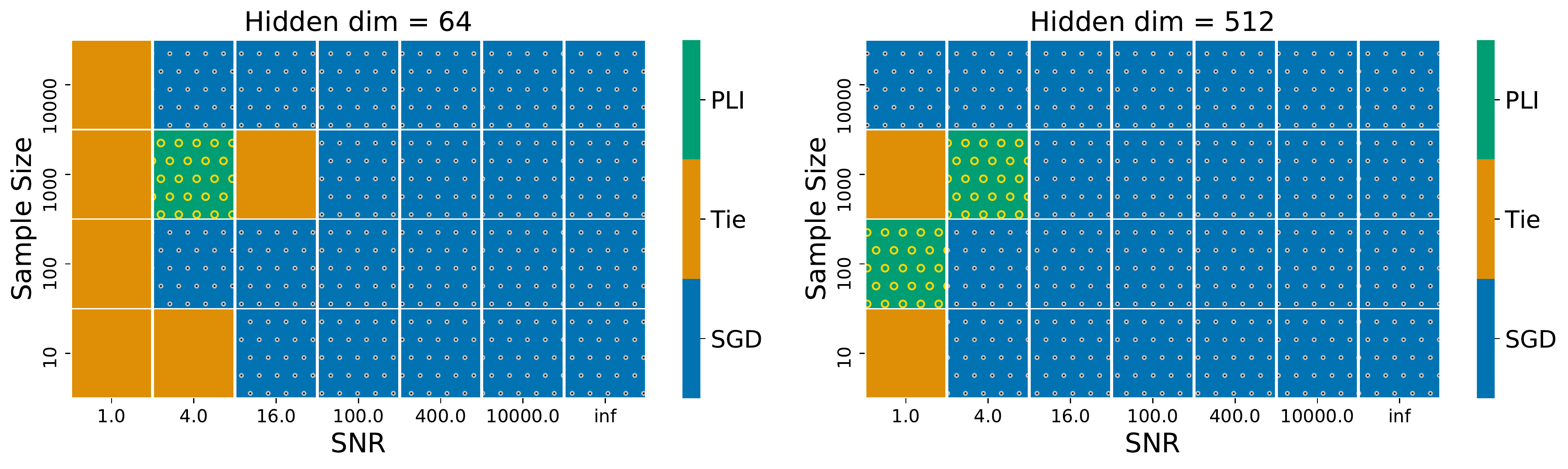}
	\end{center}	
	\caption{\small
	Synthetic multi-output regression: stochastic subgradient method (SGD) vs. the prox-linear with incremental gradient inner-loop (PLI) while varying the number of samples $\nsamp$ and the signal-to-noise ratio (SNR). We highlight the method which finds the smallest test $\ell_2$ loss after 100 epochs.
	\label{fig:pl:heatmap}}
\end{figure}

\section{Experiments}
We consider 3 setups: multi-output regression, structured prediction, and solving non-linear equations.
All hyper-parameters are tuned by grid search.

\myparagraph{Synthetic Multi-output Regression}
We consider a regression task of predicting output $\yv \in \reals^{k}$ from input $\xv \in \reals^{p}$, given a synthetic dataset  $\{(\xv_i, \yv_i)\}_{i=1}^{\nsamp}$ of input-output pairs of varying size $\nsamp$ where $p=128$ and $k=10$.
We sample each input as $\xv_i \sim \mathcal{N}(0, \Sigma)$, where the covariance $\Sigma$ exhibits a $1/j^2$ spectral decay. 
The output is generated as $\yv_i = \varphi^\star(\xv_i; \wv^\star) + \sigma\xiv_i$, where 
$\varphi^\star(\cdot\,; \wv^\star)$ is a multilayer perceptron (MLP) with one hidden layer of width $256$, and standard normal weights $\wv^\star$, while $\xiv_i$ is sampled from a standard Laplace distribution in $\reals^k$ and $\sigma$ is the noise scale, which we vary. 
We define the signal-to-noise ratio (SNR) of a problem instance as $\text{SNR} = \|\wv^\star\|^2/\sigma^2$.
Finally, the loss and evaluation measure we use is the nonsmooth $\ell_2$ loss. 

\begin{figure}[t]
	\begin{center}
		\includegraphics[width=0.7\linewidth]{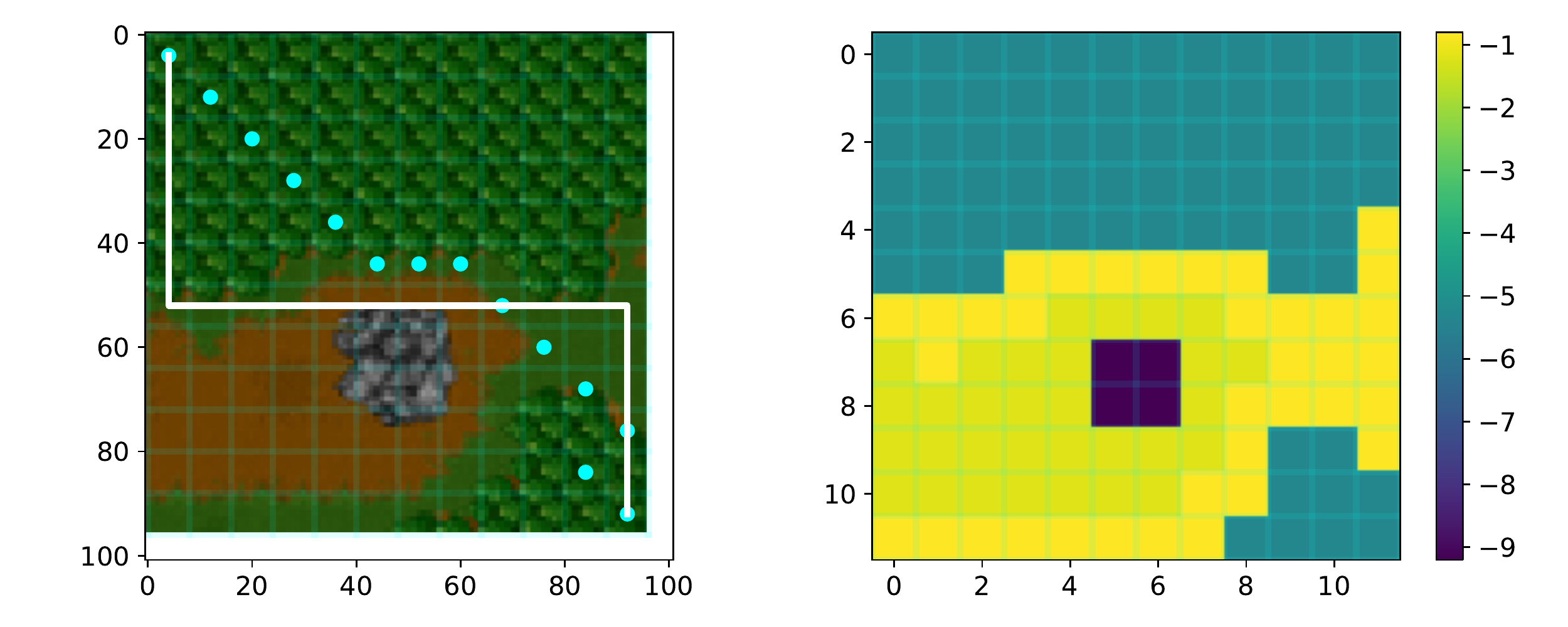}
	\end{center}	
	\caption{
	\small Planning example. \textbf{Left}: a map and its best path in (solid) white. \textbf{Right}: corresponding rewards. \label{fig:pl:warcraft}}
\end{figure}

We vary the number of samples $\nsamp$ and the SNR (equivalently, $\sigma$) and compare the two methods introduced in \Cref{sec:pl:bg}: the stochastic subgradient method (SGD) and prox-linear with incremental gradient inner loop (PLI). We tune hyperparameters to achieve the smallest loss on a held-out validation dataset in $100$ epochs and report the test loss. We run the experiment in two regimes: (a) under-parameterized, where the model is an MLP with $64$ hidden units, and, (b) over-parameterized, where the MLP has $512$ hidden units, compared to the $256$ hidden units of $\varphi^\star(\cdot\,; \wv^\star)$.
We see in \Cref{fig:pl:heatmap} that \textbf{SGD tends to outperform PLI overall}, especially in the high SNR regime. In the low SNR regime, PLI and SGD are mostly tied in their performance, exhibiting very similar test errors.

\myparagraph{Path Planning as Structured Prediction}
Among all monotonic paths from the top left corner to the bottom right corner of a grid, our task is to find the path that maximizes the rewards collected on each tile it passes through. Specifically, we consider images generated in the game Warcraft II~\cite{guyomarch2017warcraft}; see \Cref{fig:pl:warcraft}. Each tile corresponds to some terrain such as water, desert, grass, or rock with a fixed reward (grass $>$ desert $>$ water $>$ rock). 
As long as the rewards can directly be observed, this task can be solved by dynamic programming. 
In this experiment, the rewards are not directly observed; 
they are computed as the transformation of the raw pixels of each tile by a convolutional neural network. Our goal is to learn the reward function from a dataset of random maps with their associated optimal path.
Given a map $\xv$ with associated best path $\yv$, denote by $\psi(\xv, \yv, \yv'; \wv)$ the score of a path $\yv'$ parameterized by $\wv$. Our objective is
\begin{align}\label{eq:pl:planning}
	\min_{\wv\in\reals^d}\ 
	\frac{1}{n} \sum_{i=1}^n  
	\max_{\yv' \in \mathcal{Y}} 
	\psi(\xv_i, \yv_i, \yv' ; \wv)  
 + \frac{\mu}{2}\|\wv\|_2^2,
\end{align}
where $(\xv_i)_{i=1}^n$ are the maps, $(\yv_i)_{i=1}^n$ are their best paths, and $\mu \geq 0$ is a regularization parameter.

\begin{figure}[t]
	\begin{center}
		\includegraphics[width=0.8\linewidth]{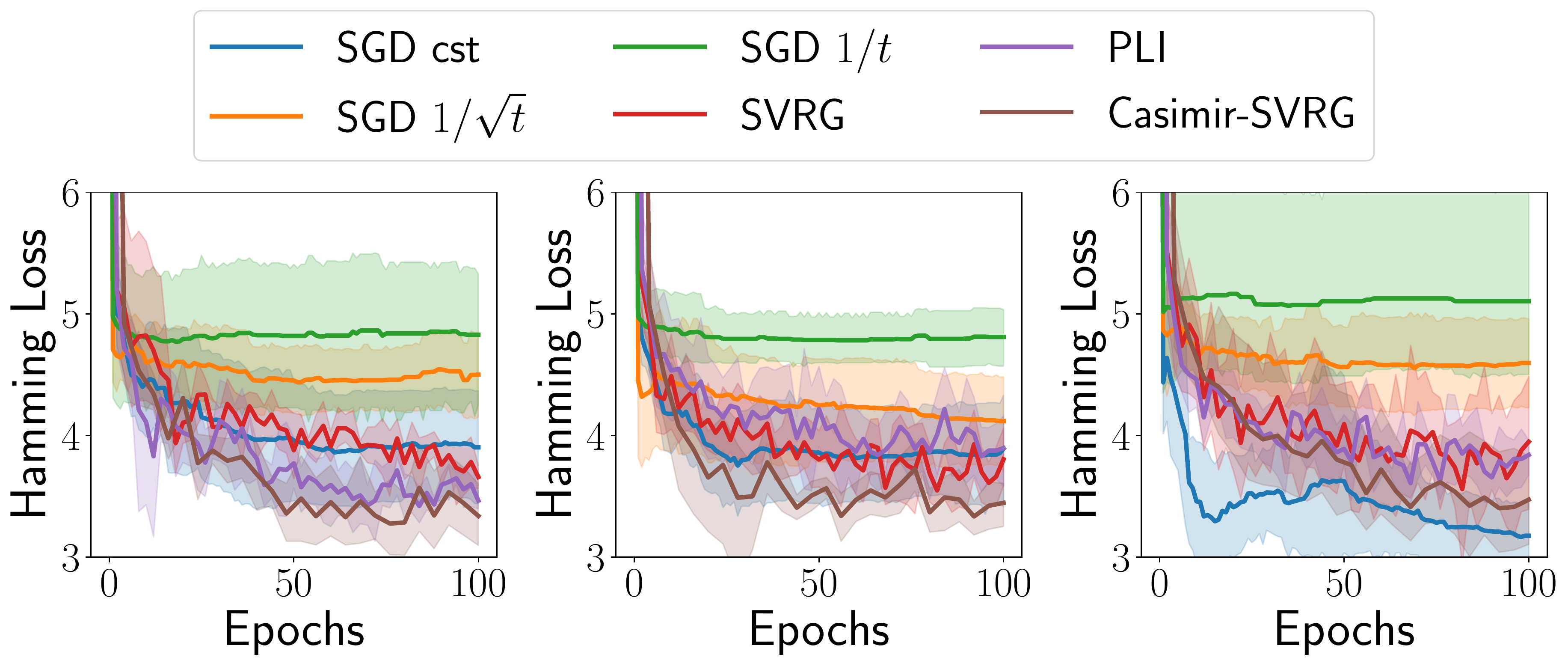}
	\end{center}
	\caption{\small Planning as a structured prediction problem.
	We plot the Hamming loss, which measures how good the predicted path is to the actual shortest path on unseen grids. From left to right: $\mu= 1/n, 10^{-2}/n, 10^{-4}/n$ \label{fig:pl:deep_warcraft}}
\end{figure}

The methods we consider are (i) stochastic subgradient methods \cite{davis2019stochastic}, denoted SGD, with various learning rates strategies
$\gamma_t = \gamma_0$, $\gamma_t = \gamma_0 /\sqrt{t}$
and $\gamma_t = \gamma_0 / t$, (ii) a variance-reduced stochastic sub-gradient method, denoted SVRG~\cite{johnson2013accelerating},  (iii) an accelerated algorithm on the Moreau-envelope of the objective as described in~\cite{pillutla2018smoother}, denoted Casimir-SVRG, (iv) a prox-linear algorithm with incremental inner loop as described in~\eqref{eq:pl:prox_lin}, denoted PLI.
For SGD, subgradients of~\eqref{eq:pl:planning} can be computed by estimating the highest reward path $y'$ associated with a given sample $(x_i, y_i)$ for a  feature map parameterized by the current parameters $w$. 
Both Casimir-SVRG and PL are implemented by smoothing the objective~\eqref{eq:pl:planning}. We take inf-convolution of the max by a squared $\ell_2$ norm, which can be approximated by returning the top-$K$ shortest paths for the given score function~\cite{pillutla2018smoother}.

In \Cref{fig:pl:deep_warcraft}, we observe that SGD with constant step-size carefully tuned can perform as well as more sophisticated methods such as the modified Gauss-Newton method. Most importantly, for a small regularization parameter ($\mu=10^{-4}/n$), \textbf{SGD yields the best test Hamming loss}, which in this task is the target metric.

\begin{figure}[t]
	\begin{center}
	\includegraphics[width=0.65\linewidth]{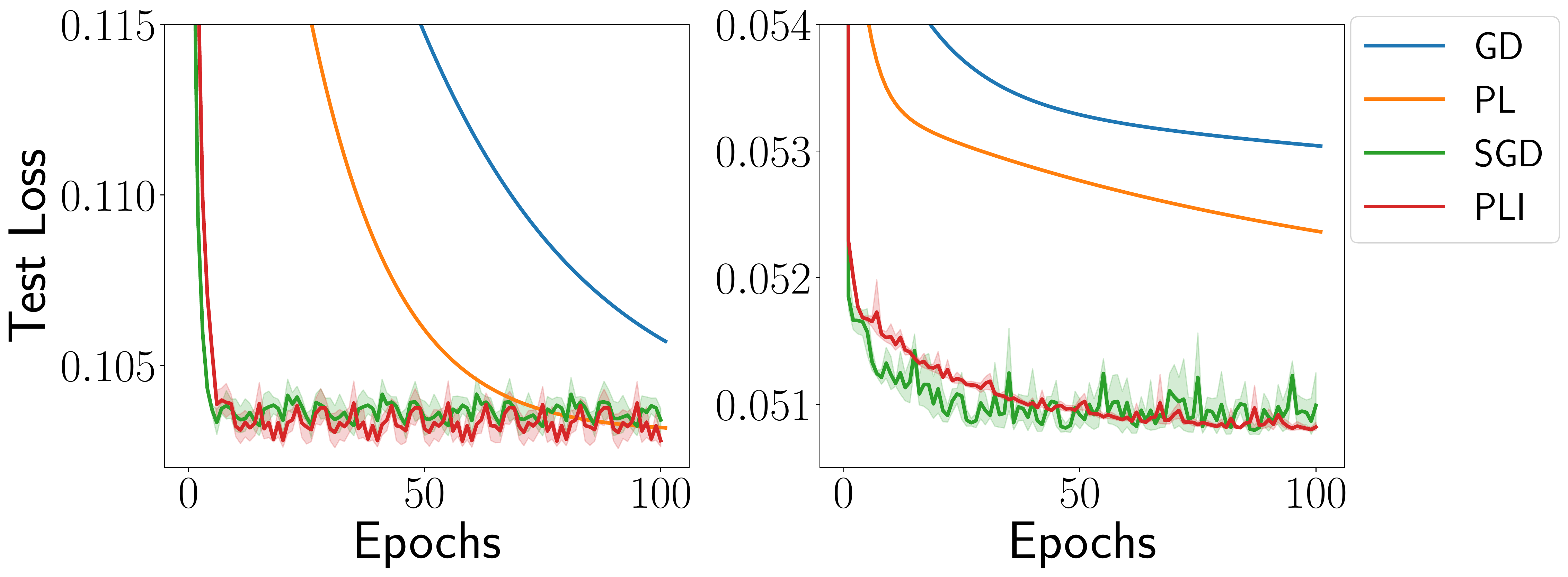}
	\caption{Solving stochastic nonlinear equations. We plot the  ``test loss'', which is the objective value on a separate testing set. \textbf{Left}: ijccn1 dataset. \textbf{Right}: covtype dataset \label{fig:stoch_non_lin}}
	\end{center}
\end{figure}

\myparagraph{Solving Non-linear Equations}
Gauss-Newton-type methods can be applied to stochastic non-linear equations 
of the form
\begin{equation}\label{eq:nonlin_pb}
\min_{\varin \in \reals^\dimin} \ \outerfunc\left( \frac{1}{\nsamp} \sum_{i=1}^\nsamp \innerfunc_i(\varin) \right),
\end{equation}
where $\outerfunc$ is a convex, possibly nonsmooth, Lipschitz function such as $\|\cdot\|_1$ and the inner mappings are smooth and typically of the form $\innerfunc_i(\varin) = \varphi(\xv_i, \varin) - \yv_i$~\cite{tran2020stochastic, zhang2020stochastic}. Problem~\eqref{eq:nonlin_pb} can be interpreted as ensuring that, on average, the non-linear mapping $\varphi(\cdot, \varin)$  maps the inputs $\xv_i$ to the targets $\yv_i$. 

\noindent {\it Algorithms.}
Denoting $\innerfunc(\varin) = \frac{1}{n} \sum_{i=1}^{n} \innerfunc_i(\varin) $, a natural baseline algorithm is to compute iterates as
\begin{equation}\label{eq:baseline_nonlin_pb}
\varin_{t+1} = \varin_t -\stepsize\widehat \nabla \innerfunc(\varin_t)\T \nabla \outerfunc(\widehat \innerfunc(\varin_t))),
\end{equation}
where $\widehat \nabla \innerfunc(\varin)$ and $\widehat \innerfunc(\varin)$ are approximations of $\nabla \innerfunc(\varin)$ and $\innerfunc(\varin)$ respectively that can be approximated from a mini-batch~\cite{wang2017stochastic}; we call this ``SGD''. Note that the minibatch subgradient estimates can be biased since the outer function $\outerfunc$ can be non-linear. 
A modified Gauss-Newton or prox-linear method adapted to the inner finite sum performs the iterations
\begin{align}
	\varin_{t+1} = \argmin_{\varin \in \reals^\dimin} \  & \Big\{ \outerfunc \left( \widehat \innerfunc(\varin_t)  +  \widehat \nabla \innerfunc(\varin_t) (\varin - \varin_t)\right)
	+ \frac{\kappa}{2} \|\varin-\varin_t \|_2^2 \Big\},
 \label{eq:sgn}
\end{align}
where each sub-problem can be solved by incremental algorithms such as the accelerated dual proximal gradient ascent~\cite{tran2020stochastic, zhang2020stochastic}.

\noindent {\it Experiment.}
We consider the experimental setting of~\cite{tran2020stochastic}. The objective is to solve~\eqref{eq:nonlin_pb}
where $\outerfunc$ is the Huber loss, a smooth surrogate of the nonsmooth $\ell_1$ norm,
and inner mappings $\innerfunc$ are the concatenation of four different losses, i.e., $\innerfunc_i(\varin) = (\ell_1(\xv_i^\top \varin, \yv_i), \ldots, \ell_4(\xv_i^\top \varin, \yv_i))$, where the formulations of the losses can be found in~\cite{tran2020stochastic}.
The samples $(\xv_i, \yv_i)$ are drawn from the datasets {\texttt{\small ijcnn1}} or \texttt{\small covtype} from the LIBSVM repository~\cite{libsvm}.

We consider (i) a gradient descent denoted GD, (ii) a modified Gauss-Newton or prox-linear method denoted PL, (iii) a baseline of the form~\eqref{eq:baseline_nonlin_pb}, denoted SGD, (iv) an incremental Gauss-Newton or prox-linear method as described in~\eqref{eq:sgn}, denoted PLI for consistency. 
In \Cref{fig:stoch_non_lin}, we observe that PL outperforms GD as expected in the batch setting. However, this advantage is no longer present in the incremental setting. Here, we find that
the \textbf{SGD baseline~\eqref{eq:baseline_nonlin_pb}
performs on par with the Gauss-Newton variant PLI}.

{\small
\bibliography{prox-linear}

\begin{thebibliography}{30}
\providecommand{\natexlab}[1]{#1}
\providecommand{\url}[1]{\texttt{#1}}
\expandafter\ifx\csname urlstyle\endcsname\relax
  \providecommand{\doi}[1]{doi: #1}\else
  \providecommand{\doi}{doi: \begingroup \urlstyle{rm}\Url}\fi

\bibitem[Bj{\"o}rck(1996)]{bjorck1996numerical}
{\AA}.~Bj{\"o}rck.
\newblock \emph{Numerical methods for least squares problems}.
\newblock SIAM, 1996.

\bibitem[Burke(1985)]{burke1985descent}
J.~V. Burke.
\newblock Descent methods for composite nondifferentiable optimization
  problems.
\newblock \emph{Mathematical Programming}, 33\penalty0 (3):\penalty0 260--279,
  1985.

\bibitem[Burke and Ferris(1995)]{burke1995gauss}
J.~V. Burke and M.~C. Ferris.
\newblock A {Gauss-Newton} method for convex composite optimization.
\newblock \emph{Mathematical Programming}, 71\penalty0 (2):\penalty0 179--194,
  1995.

\bibitem[Chang and Lin(2011)]{libsvm}
C.-C. Chang and C.-J. Lin.
\newblock {LIBSVM}: A library for support vector machines.
\newblock \emph{ACM Transactions on Intelligent Systems and Technology},
  2:\penalty0 27:1--27:27, 2011.
\newblock Software available at \url{http://www.csie.ntu.edu.tw/~cjlin/libsvm}.

\bibitem[Cichocki and Amari(2002)]{cichocki2002adaptive}
A.~Cichocki and S.-i. Amari.
\newblock \emph{{Adaptive blind signal and image processing: learning
  algorithms and applications}}.
\newblock John Wiley \& Sons, 2002.

\bibitem[Crammer and Singer(2001)]{crammer2001algorithmic}
K.~Crammer and Y.~Singer.
\newblock On the algorithmic implementation of multiclass kernel-based vector
  machines.
\newblock \emph{Journal of Machine Learning Research}, 2:\penalty0 265--292,
  2001.

\bibitem[Davis and Drusvyatskiy(2019)]{davis2019stochastic}
D.~Davis and D.~Drusvyatskiy.
\newblock Stochastic model-based minimization of weakly convex functions.
\newblock \emph{SIAM Journal on Optimization}, 29\penalty0 (1):\penalty0
  207--239, 2019.

\bibitem[Drusvyatskiy and Paquette(2019)]{drusvyatskiy2019efficiency}
D.~Drusvyatskiy and C.~Paquette.
\newblock Efficiency of minimizing compositions of convex functions and smooth
  maps.
\newblock \emph{Mathematical Programming}, 178\penalty0 (1):\penalty0 503--558,
  2019.

\bibitem[Gales et~al.(2012)Gales, Watanabe, and
  Fosler-Lussier]{gales2012structured}
M.~J.~F. Gales, S.~Watanabe, and E.~Fosler-Lussier.
\newblock Structured discriminative models for speech recognition: An overview.
\newblock \emph{IEEE Signal Processing Magazine}, 29\penalty0 (6):\penalty0
  70--81, 2012.

\bibitem[Guyomarch(2017)]{guyomarch2017warcraft}
J.~Guyomarch.
\newblock Warcraft {II} open-source map-editor, 2017.
\newblock URL \url{http://github.com/ war2/war2edit}.

\bibitem[Herring et~al.(2019)Herring, Nagy, and Ruthotto]{herring2019gauss}
J.~L. Herring, J.~Nagy, and L.~Ruthotto.
\newblock {Gauss--Newton} optimization for phase recovery from the bispectrum.
\newblock \emph{IEEE Transactions on Computational Imaging}, 6:\penalty0
  235--247, 2019.

\bibitem[Huang and Fu(2019)]{huang2019low}
K.~Huang and X.~Fu.
\newblock {Low-Complexity Proximal Gauss-Newton Algorithm for Nonnegative
  Matrix Factorization}.
\newblock In \emph{2019 IEEE Global Conference on Signal and Information
  Processing (GlobalSIP)}, pages 1--5. IEEE, 2019.

\bibitem[Johnson and Zhang(2013)]{johnson2013accelerating}
R.~Johnson and T.~Zhang.
\newblock Accelerating stochastic gradient descent using predictive variance
  reduction.
\newblock \emph{Advances in Neural Information Processing Systems}, 2013.

\bibitem[Laurent and Massart(2000)]{laurent2000adaptive}
B.~Laurent and P.~Massart.
\newblock Adaptive estimation of a quadratic functional by model selection.
\newblock \emph{Annals of Statistics}, pages 1302--1338, 2000.

\bibitem[Levenberg(1944)]{levenberg1944method}
K.~Levenberg.
\newblock A method for the solution of certain non-linear problems in least
  squares.
\newblock \emph{Quarterly of applied mathematics}, 2\penalty0 (2):\penalty0
  164--168, 1944.

\bibitem[Nesterov(2007)]{nesterov2007modified}
Y.~Nesterov.
\newblock Modified {Gauss--Newton} scheme with worst case guarantees for global
  performance.
\newblock \emph{Optimisation methods and software}, 22\penalty0 (3):\penalty0
  469--483, 2007.

\bibitem[Nocedal and Wright(2006)]{nocedal2006numerical}
J.~Nocedal and S.~Wright.
\newblock \emph{{Numerical Optimization}}.
\newblock Springer Science \& Business Media, 2006.

\bibitem[Pillutla et~al.(2018)Pillutla, Roulet, Kakade, and
  Harchaoui]{pillutla2018smoother}
K.~Pillutla, V.~Roulet, S.~M. Kakade, and Z.~Harchaoui.
\newblock A smoother way to train structured prediction models.
\newblock In \emph{Advances in Neural Information Processing Systems}, 2018.

\bibitem[Ratliff et~al.(2006)Ratliff, Bagnell, and
  Zinkevich]{ratliff2006maximum}
N.~D. Ratliff, J.~A. Bagnell, and M.~Zinkevich.
\newblock Maximum margin planning.
\newblock In \emph{International Conference Machine Learning}, 2006.

\bibitem[Repetti et~al.(2014)Repetti, Chouzenoux, and
  Pesquet]{repetti2014nonconvex}
A.~Repetti, E.~Chouzenoux, and J.-C. Pesquet.
\newblock A nonconvex regularized approach for phase retrieval.
\newblock In \emph{2014 IEEE International Conference on Image Processing
  (ICIP)}, pages 1753--1757. IEEE, 2014.

\bibitem[Rockafellar and Wets(2009)]{rockafellar2009variational}
R.~T. Rockafellar and R.~J.-B. Wets.
\newblock \emph{{Variational Analysis}}, volume 317.
\newblock Springer Science \& Business Media, 2009.

\bibitem[Roulet et~al.(2019)Roulet, Drusvyatskiy, Srinivasa, and
  Harchaoui]{roulet2019iterative}
V.~Roulet, D.~Drusvyatskiy, S.~S. Srinivasa, and Z.~Harchaoui.
\newblock Iterative linearized control: Stable algorithms and complexity
  guarantees.
\newblock In \emph{International Conference on Machine Learning}, 2019.

\bibitem[Rush(2020)]{rush-2020-torch}
A.~Rush.
\newblock Torch-struct: Deep structured prediction library.
\newblock In \emph{Annual Meeting of the Association for Computational
  Linguistics: System Demonstrations}, pages 335--342, 2020.
\newblock \doi{10.18653/v1/2020.acl-demos.38}.

\bibitem[Sideris and Bobrow(2005)]{sideris2005efficient}
A.~Sideris and J.~E. Bobrow.
\newblock An efficient sequential linear quadratic algorithm for solving
  nonlinear optimal control problems.
\newblock In \emph{Proceedings of the 2005, American Control Conference,
  2005.}, pages 2275--2280. IEEE, 2005.

\bibitem[Taskar et~al.(2004)Taskar, Guestrin, and Koller]{taskar2004max}
B.~Taskar, C.~Guestrin, and D.~Koller.
\newblock Max-margin {M}arkov networks.
\newblock In \emph{Advances in Neural Information Processing Systems}, pages
  25--32, 2004.

\bibitem[Tran-Dinh et~al.(2020)Tran-Dinh, Pham, and Nguyen]{tran2020stochastic}
Q.~Tran-Dinh, N.~Pham, and L.~Nguyen.
\newblock Stochastic gauss-newton algorithms for nonconvex compositional
  optimization.
\newblock In \emph{International Conference on Machine Learning}, 2020.

\bibitem[Tsochantaridis et~al.(2004)Tsochantaridis, Hofmann, Joachims, and
  Altun]{tsochantaridis2004support}
I.~Tsochantaridis, T.~Hofmann, T.~Joachims, and Y.~Altun.
\newblock Support vector machine learning for interdependent and structured
  output spaces.
\newblock In \emph{International Conference on Machine Learning}, 2004.

\bibitem[Viterbi(1967)]{viterbi1967error}
A.~J. Viterbi.
\newblock Error bounds for convolutional codes and an asymptotically optimum
  decoding algorithm.
\newblock \emph{{IEEE} Trans. Information Theory}, 13\penalty0 (2):\penalty0
  260--269, 1967.
\newblock \doi{10.1109/TIT.1967.1054010}.

\bibitem[Wang et~al.(2017)Wang, Fang, and Liu]{wang2017stochastic}
M.~Wang, E.~X. Fang, and H.~Liu.
\newblock Stochastic compositional gradient descent: Algorithms for minimizing
  compositions of expected-value functions.
\newblock \emph{Mathematical Programming}, 161\penalty0 (1-2):\penalty0
  419--449, 2017.

\bibitem[Zhang and Xiao(2022)]{zhang2020stochastic}
J.~Zhang and L.~Xiao.
\newblock Stochastic variance-reduced prox-linear algorithms for nonconvex
  composite optimization.
\newblock \emph{Mathematical Programming}, 195\penalty0 (1):\penalty0 649--691,
  2022.

\end{thebibliography}
\bibliographystyle{abbrvnat}
}
\clearpage
\appendix

\section{Proofs}
\label{sec:a:cas:quadratic}

\subsection{Proof of Quadratic Convergence (\Cref{prop:pl:main})}

Here, we give the full statement and a simple proof of \Cref{prop:pl:main}, following standard techniques~\cite[Theorem 3]{nesterov2007modified}. 
We prove the proposition for the case $n = 1$ first by considering the problem $\min_{w \in \reals^\dimin} \outerfunc \circ \innerfunc(w)$. We then generalize to $n > 1$ for a proof of \Cref{prop:pl:main} in full generality. 

We are interested primarily in the overparameterized case where $\dimout \le \dimin$. 
Below, we denote  $\grad \innerfunc(\varin) \in \mathbb{R}^{\dimout \times \dimin}$ as the Jacobian of $\innerfunc$ at $\varin$. 
We impose the assumption that the minimal singular value of the transposed Jacobian is bounded away from $0$ as 
$\sigma_{\min}(\grad \innerfunc(\varin)\T) \geq \surj >0$ for any $\varin \in \reals^\dimin$, 
This assumption implies the surjectivity of the Jacobian at each $\varin$. That is, for every $\varout \in \reals^\dimout$, there exists a $v \in \reals^\dimin$ such that $\grad \innerfunc(\varin) v = \varout$.

We also assume that the following minimum values are reached: 
\[
    \outerfunc^* = \min_{\varout \in \reals^\dimout} \outerfunc(\varout),
    \quad \text{and}, \quad
    (\outerfunc\circ \innerfunc)^* = \min_{\varin \in \reals^\dimin} \outerfunc\big(\innerfunc(\varin)\big).
\]

We have the following statement. 
\begin{proposition} \label{prop:pl:combined}
	Consider the compositional problem $\min_{\varin \in \reals^\dimin} \outerfunc \circ \innerfunc(\varin)$ with the following assumptions:   
	\begin{enumerate}[label=(\alph*)]
	    \item $\outerfunc$ is $\ell$-Lipschitz continuous, convex and $\sharparam$-sharp, i.e.,  $\outerfunc(\varout) - \outerfunc^* \geq \sharparam \,  \dist(\varout, \spaceout^*)$ for any $\varout \in \reals^\dimout$ with $\sharparam>0$ and $\dist(\varout, \spaceout^*)$ the Euclidean distance of $\varout$ to $\spaceout^* = \argmin_{\varout\in \reals^\dimout} \outerfunc(\varout) \neq \emptyset$. 
	    \item $\innerfunc$ is $L$-smooth and satisfies $\sigma_{\min}(\grad \innerfunc(\varin)\T) \geq \surj >0$ for any $\varin \in \reals^\dimin$. 
	\end{enumerate}
	The sequence $(\varin_t)_{t=0}^\infty$ produced by the prox-linear algorithm~\eqref{eq:pl:prox_lin} with $M = L \ell$ starting from arbitrary $\varin_0 \in \reals^\dimin$ converges 
	globally as $(\outerfunc \circ \innerfunc)(\varin_t) \to  (\outerfunc \circ \innerfunc)^* = \outerfunc^*$.
	Furthermore, as soon as an iterate $\varin_{j}$ satisfies $\outerfunc(\innerfunc(\varin_t)) - (\outerfunc\circ\innerfunc)^* \leq (\sharparam\surj)^2/(L\ell)
	$, the subsequence $(w_t)_{t=j}^\infty$ convergences quadratically as
	\[
        \outerfunc(\innerfunc(\varin_{t+1})) -(\outerfunc\circ\innerfunc)^*
        \leq \frac{L\ell}{2(\sharparam\surj)^2}(\outerfunc(\innerfunc(\varin_t)) - (\outerfunc\circ\innerfunc)^*) ^2 
        \leq \frac{1}{2} (\outerfunc(\innerfunc(\varin_t)) - (\outerfunc\circ\innerfunc)^*).
	\]
\end{proposition}
\begin{proof} 
	For an iterate $\varin_t$ of the prox-linear algorithm, denote $\varout_t^* = \operatorname{Proj}_{\spaceout^*}(\innerfunc(\varin_t))$ the Euclidean projection of $\innerfunc(\varin_t)$ onto the set of minimizers of $\outerfunc$ such that $\dist(\innerfunc(\varin_t), \spaceout^*) = \|\innerfunc(\varin_t)- \varout_t^*\|_2$.
	
	Since the Jacobian $\grad \innerfunc(\varin_t)$ is surjective, there exists $v_t^*$ be such that $\grad \innerfunc(\varin_t) v_t^*  = \varout^*_t - \innerfunc(\varin_t)$.  
	Furthermore, from the minimum singular value condition, there exists a choice of $v_t^*$ such that $\|v_t^*\|\leq \|\varout^*_t - \innerfunc(\varin)\|_2/\surj$ (see \cite[Lemma 6]{nesterov2007modified} for a proof).
	
	If $M \geq L\ell$, then the iterates of the prox-linear algorithm satisfy~\cite{drusvyatskiy2019efficiency}
	\begin{align}
		\outerfunc(\innerfunc(\varin_{t+1})) &\leq \min_{v\in \reals^\dimin} \left\{ \outerfunc(\innerfunc(\varin_t) + \nabla\innerfunc(\varin_t) v)  + \frac{\kappa}{2} \|v\|_2^2  
		\right\} \nonumber\\
		& \stackrel{\text{(i)}}{\leq} \min_{s \in [0,1]} 
		\left\{ \outerfunc(\innerfunc(\varin_t) + s \grad \innerfunc(\varin_t) v_t^*) + \frac{\kappa s^2}{2} \|v_t^*\|_2^2 
		\right\} \nonumber\\
		& \stackrel{\text{(ii)}}{\leq} \min_{s\in [0, 1]} \left\{ \outerfunc(\innerfunc(\varin_t) + s (\varout_t^* -\innerfunc(\varin_t)) + \frac{\kappa s^2}{2\surj^2} \|\varout_t^*-\innerfunc(\varin_t)\|_2^2  \right\} \nonumber\\
		& \stackrel{\text{(iii)}}{\leq} \min_{s\in [0, 1]} \left\{ s \outerfunc^* + (1-s)\outerfunc(\innerfunc(\varin_t)) + \frac{\kappa s^2}{2(\surj\sharparam)^2} (\outerfunc(\innerfunc(\varin_t)) - \outerfunc^*)^2 \right\} \,.\label{eq:pl:prox_lin_proof1}
	\end{align}
	Here, we (i) restricted the domain of the minimization to $v = s v_t^*$ with $s \in [0, 1]$, (ii) plugged in the definition of $v_t^*$ and the bound on $\|v_t^*\|$, and, 
	(iii) used the convexity and sharpness of $\outerfunc$. 
	Next, by subtracting $\outerfunc^*$ from both sides, we get
	\begin{align*}
		\outerfunc(\innerfunc(\varin_{t+1})) - \outerfunc^*  
		& \leq 	\min_{s \in [0,1]} 
		\left\{ (1-s)(\outerfunc(\innerfunc(\varin_t)) - \outerfunc^*) 
		+ \frac{s^2\kappa }{2 (\surj\sharparam)^2}(\outerfunc(\innerfunc(\varin_t)) - \outerfunc^*) ^2 \right\} . 
	\end{align*}
	If $\outerfunc(\innerfunc(\varin_t)) - \outerfunc^* \leq (\sharparam\surj)^2/\kappa$, the minimum in~\eqref{eq:pl:prox_lin_proof1} is reached at $s=1$ and we get 
	\[
	\outerfunc(\innerfunc(\varin_{k+1})) - \outerfunc^* \leq \frac{\kappa }{2 (\surj\sharparam)^2}(\outerfunc(\innerfunc(\varin_t)) - \outerfunc^*) ^2 \leq \frac{1}{2} (\outerfunc(\innerfunc(\varin_t)) - \outerfunc^*).
	\]
	This is the quadratic convergence phase. 
	On the other hand, if $\outerfunc(\innerfunc(\varin_t)) - \outerfunc^* \geq (\sharparam\surj)^2/\kappa$, then the minimum in~\eqref{eq:pl:prox_lin_proof1} is reached at 
	$
	    s = (\surj\sharparam)^2 /\left(\kappa \left(\outerfunc(\innerfunc(\varin_t) - \outerfunc^* \right)\right)
	$
, and we have the bound
	\[
	\outerfunc(\innerfunc(\varin_{t+1})) - \outerfunc^* \leq \outerfunc(\innerfunc(\varin_t)) - \outerfunc^* -\frac{(\sharparam\surj)^2}{2\kappa}.
	\]
	Since $f$ is bounded from below, the sequence $\outerfunc(\innerfunc(\varin_t))$ converges to $\outerfunc^*$. Hence, the minimum of the composite objective matches the minimum of the outer function, i.e., $\outerfunc^* = (\outerfunc\circ\innerfunc)^*$.
\end{proof}

We can now prove \Cref{prop:pl:main} as a corollary of \Cref{prop:pl:combined}. 
\begin{proof}[Proof of \Cref{prop:pl:main}]
    Consider the reduction $\innerfunc: \reals^{\dimin} \to \reals^{\dimout  n}$ and $\bar \outerfunc : \reals^{\dimout n} \to \reals$ given by
    \begin{equation}\label{eq:pl:erm_decomp}
    \innerfunc(\varin) = (\innerfunc_1(\varin) ;\ldots;\innerfunc_n(\varin)),
    \quad\text{and},\quad 
    \bar\outerfunc\big(\varout_1;\ldots;\varout_n\big) = \frac{1}{n}\sum_{i=1}^{n} \outerfunc_i(\varout_i) \,,
    \end{equation}
    where we use semi-colons to denote the concatenation of vectors. The finite-sum problem \eqref{eq:pl:erm_pb} now reduces to $\min_\varin \bar \outerfunc \circ \innerfunc(\varin)$. 
    
    We have by definition that $\bar \outerfunc$ is convex.
    Next, $\bar \outerfunc$ is $\bar\ell$-Lipschitz with $\bar \ell = \ell/\sqrt{n}$ since 
    \begin{align*}
        |\bar\outerfunc(u) - \bar\outerfunc(u')|
        &\le \frac{1}{n}\sum_{i=1}^n |\outerfunc(u_i) - \outerfunc(u_i')|
        \le \frac{\ell}{n}\sum_{i=1}^n \norm{u_i - u_i'}_2 
        \\
        &\le \frac{\ell}{\sqrt{n}} \left( \sum_{i=1}^n \norm{u_i - u_i'}_2^2 \right)^{1/2}
        = \frac{\ell}{\sqrt{n}} \norm{u - u'}_2 \,.
    \end{align*}
    Further, we argue that $\bar\outerfunc$ is $\bar \mu$-sharp with $\bar \mu = \mu/n$.
    Note that $(U^*)^n$ is the argmin set of $\bar \outerfunc$ where $U^*$ is the argmin set of $\outerfunc$. Further, their minimum values satisfy $\bar \outerfunc^* := \min \bar \outerfunc = \min \outerfunc = f^*$. Therefore, we have, 
    \begin{align*}
        \bar \outerfunc(u) - \bar \outerfunc^*
        &=
        \frac{1}{n} \sum_{i=1}^n (\outerfunc(u_i) - \outerfunc^*)
        \ge \frac{\mu}{n} \sum_{i=1}^n \dist(u_i, U^*) 
        \ge \frac{\mu}{n} \dist(u, (U^*)^n) ,
    \end{align*}
    where we used 
    \[
        \dist(u, (U^*)^n) = \sqrt{\sum_{i=1}^n \dist(u_i, U^*)^2}
        \le \sum_{i=1}^n \dist(u_i, U^*). 
    \]
    All the assumptions of \Cref{prop:pl:combined} are met, and invoking it now completes the proof.
\end{proof}

\subsection{Proof of the Statistical Setting (\Cref{prop:pl:noise})}

We give the full statement of \Cref{prop:pl:noise} and its proof.

\begin{proposition} \label{prop:pl:noise:full}
    Fix some $\delta \in (0, 1)$ and consider problem \eqref{eq:pl:erm_pb} with $\gv_i$ and $\outerfunc$ as defined above with the output dimension $\dimout \ge 4 \log (2n/\delta)$. Suppose that
    (i)
    $\varin \mapsto \varphi(\xv_i; \varin)$ is $L$-smooth for each $i \in [n]$, and, 
    (ii) the function $\varphi$ can interpolate the data so that $\varphi(\xv_i; \wv^*) = \yv_i$ for each $i \in [n]$ for some $\varin^* \in \reals^\dimin$.
    Suppose there exists a scalar $R>0$ and an integer $j$ such that for all integers $t \ge j$, we have
    \begin{align} \label{eq:pl:noise-quadratic-a}
        F(\wv_t) - \min F \leq R\,, \quad
        \text{and} \quad
        F(\wv_{t+1}) - \min F \leq \frac{1}{2R}\big(F(\wv_t) - \min F\big)^2  \,.
    \end{align}
    Then, we have the following with probability at least $1-\delta$: 
    \begin{enumerate}[label=(\alph*)]
    \item If the noise level satisfies
        \[
        \sigma > 
        \frac{R}{\sqrt{\dimout}} \left( 1 - \left( \frac{4}{\dimout} \log (2n/\delta) \right)^{1/4}  \right)^{-1} \,,
    \]
    then the first iterate $\varin_t$ enjoying quadratic convergence \eqref{eq:pl:noise-quadratic-a} 
    satisfies $F(\varin_t) < F(\overline \varin)$.
    \item Conversely, if the noise level satisfies
        \[
        \sigma <
        \frac{R}{\sqrt{\dimout}} \left( 1 + \left( \frac{16}{\dimout} \log (2n/\delta) \right)^{1/4}  \right)^{-1} \,,
    \]
    then the first iterate $\varin_t$ enjoying quadratic convergence \eqref{eq:pl:noise-quadratic-a} 
    satisfies $F(\varin_t) > F(\overline \varin)$. 
    \end{enumerate}
\end{proposition}
\begin{proof} 
First, under the interpolation assumption, we have that $0 \le \min F \le F(\varin^*) = 0$, so $\min F = 0$.
Therefore, in this setting, 
quadratic convergence holds before the noise level if and only if $F(\overline \varin) \ge R$. 
To complete the proof, we show below that with probability at least $1-\delta$ that
\begin{align} \label{eq:pl:stat-arg-0}
   \sigma \sqrt{\dimout} \left( 1 - \left( \frac{4}{\dimout} \log (2n/\delta) \right)^{1/4}  \right)
   \le  F(\overline \varin) \le 
   \sigma \sqrt{\dimout} \left( 1 + \left( \frac{16}{\dimout} \log (2n/\delta) \right)^{1/4}  \right)\,.
\end{align}

To this end, we simplify
    \begin{align} \label{eq:pl:prox-lin:stat-arg}
        F(\overline \varin) = \frac{1}{n} \sum_{i=1}^n \| \varphi(\xv_i ; \overline \varin) - \yv_i \|_2 
        = \frac{1}{n} \sum_{i=1}^n \| \xiv_i\|_2 \,. 
    \end{align}
    Noting that $\|\xiv_i\|_2^2$ follows a $\chi^2$ distribution with $\dimout$ degrees of freedom, a standard concentration argument shows that  (see example \cite[Lemma 1]{laurent2000adaptive})
    \[
        \mathbb{P}\left(
        \sigma^2 \dimout \left(1 - 2\sqrt{\frac{\lambda}{\dimout}} \right)
        \le \|\xiv_i\|^2_2 
        \le
        \sigma^2 \dimout \left(1 + 2\sqrt{\frac{\lambda}{\dimout}} + \frac{2\lambda}{\dimout}\right)
        \right) \ge 1 - 2\exp(-\lambda) \,
    \]
    for any $\lambda > 0$. 
    Next, we plug in $\lambda = \log (2n / \delta)$.
    Noting that $\lambda / k \le 1/4$, we use the bound $\lambda/k \le \sqrt{\lambda/k}$ and $\sqrt{1-x} \ge 1-\sqrt{x}$ to get that
    \[
        \sigma \sqrt{k} \, \Big(1 - (4\lambda/k)^{1/4}\Big)
        \le \norm{\xi_i}_2 
        \le \sigma\sqrt{k}
        \sqrt{1 + 4 \sqrt{\lambda/k}}
        \le \sigma\sqrt{k} \, \Big(1 + (16\lambda/k)^{1/4}\Big) 
    \]
    with probability at least $1-\delta/n$.
    Invoking the union bound over $i=1,\cdots,n$, we have with probability at least $1-\delta$ that 
    \[
         \sigma \sqrt{\dimout} \left( 1 - \left( \frac{4}{\dimout} \log (2n/\delta) \right)^{1/4}  \right)
   \le  \norm{\xi_i}_2 \le 
   \sigma \sqrt{\dimout} \left( 1 + \left( \frac{16}{\dimout} \log (2n/\delta) \right)^{1/4}  \right)
    \]
    holds simultaneously for each $i \in [n]$. 
    Plugging this into \eqref{eq:pl:prox-lin:stat-arg} completes the proof.
\end{proof}

\end{document}